\documentclass[12pt]{amsart}
\usepackage[active]{srcltx}
\usepackage{a4wide}
\usepackage{amsthm,amsfonts,amsmath,mathrsfs,amssymb}
\usepackage{dsfont}
\usepackage[utf8]{inputenc}
\usepackage{fourier}
\def\a{\alpha}		 		  
		 \def\la{\lambda}	  
		 		  
		 		  \def\z{\zeta}
\def\ch{\chi}		 		  \def\vf{\varphi}

	 	  \def\vf{\varphi}

		 	  \def\La{\Lambda}

\def\D{{\mathbb D}}	 
\def\C{{\mathbb C}}	 \def\N{{\mathbb N}}
  
\def\ca{{\mathcal A}}	 \def\cb{{\mathcal B}}
\def\cc{{\mathcal C}}	 
\def\ch{{\mathcal H}}	 
	 \def\cm{{\mathcal M}}

\def\({\left(}		 \def\){\right)}
\def\pd{\partial}

\newtheorem{prop}{\sc Proposition}
\newtheorem{lem}[prop]{\sc Lemma}
\newtheorem{thm}[prop]{\sc Theorem}

\newtheorem{other}{\sc Theorem}				 % Other papers' theorems
\newtheorem{ex}[other]{\sc Example}     % Examples get letters
\begin{document}
%%% Title
\title[Weighted composition operators in functional Banach spaces]{Weighted composition operators in functional Banach spaces: an axiomatic approach}
%%% Information for the first author
\author[I. Ar\'evalo]{Irina Ar\'evalo}
\address{Departamento de Matem\'aticas, Universidad Aut\'onoma de
Madrid, 28049 Madrid, Spain}
\email{irina.arevalobarco@gmail.com}
%%% Information for the second author
\author[D. Vukoti\'c]{Dragan Vukoti\'c}
\address{Departamento de Matem\'aticas, Universidad Aut\'onoma de
Madrid, 28049 Madrid, Spain} \email{dragan.vukotic@uam.es}
%\dedicatory{}
\thanks{The authors were supported by MTM2015-65792-P and partially by MTM2017-90584-REDT, MINECO, Spain.}
%%% General info
\date{22 October, 2018.}
%%%%%%%%%%%%%%%%%%%%%%%%%%%
%%% Abstract
%%%%%%%%%%%%%%%%%%%%%%%%%%%
\begin{abstract}
We work with very general Banach spaces of analytic functions in the disk or other domains which satisfy a minimum number of natural axioms. Among the preliminary results, we discuss some implications of the basic axioms and identify all functional Banach spaces in which every bounded analytic function is a pointwise multiplier. Next, we characterize (in various ways) the weighted composition operators among the bounded operators on such spaces, thus generalizing some well-known results on multiplication or composition operators. We also characterize the invertible weighted composition operators on the disk and on general Banach spaces of analytic functions on bounded domains under different sets of axioms whose connections we discuss by providing appropriate examples. This generalizes and complements various recent results by Gunatillake, Bourdon, and Hyvärinen-Lindström-Nieminen-Saukko.
\end{abstract}
\maketitle
%%%%%%%%%%%%%%%%%%%%%%%%%%%
\section*{Introduction}
 \label{sec-intro}
%%%%%%%%%%%%%%%%%%%%%%%%%%%
\par
For a function $F$ analytic in a planar domain $\Omega$ and an analytic map $\vf$ of $\Omega$ into itself, the \textit{weighted composition operator\/} (sometimes abbreviated as WCO) $T_{F,\vf}$ is defined formally by the formula $T_{F,\vf}f = F (f\circ\vf) = M_F C_\vf f$ as the composition followed by multiplication; it may or may not be bounded on a given space of analytic functions. Such operators are relevant in the theory of function spaces and operator theory for various reasons:
\begin{itemize}
 \item
They generalize both the pointwise multiplication operators $M_F$ (when $\vf(z)=z$) and the composition operators $C_\vf$ (when $F\equiv 1$);
 \item
The only surjective isometries of some classical spaces of analytic functions are precisely of this type (\textit{cf.\/} \cite{Fo1, Fo2}, \cite{Kol});
 \item
Some classical operators such as the Ces\`aro operator \cite{CS} and the Hilbert matrix operator \cite{DiS} can be written as average values of weighted composition operators;
 \item
Close connections have been found between the boundedness \cite{Shim} or compactness \cite{Sm} of certain weighted composition operators on weighted Bergman spaces and the well-known  Brennan conjecture in geometric function theory.
\end{itemize}
The earliest references on WCOs appear to be \cite{Ka} and \cite{Ki}. Such operators have been studied more recently in a number of papers, \textit{e.g.}, in \cite{CHD1, CHD2}, \cite{CZ}, \cite{M}, \cite{OZ}, \cite{CGP}. The point of view of the present paper is to study the properties of WCOs in a unified way by working with general Banach spaces of analytic functions that satisfy only a handful of axioms while still obtaining that same conclusions as in the known special situations, thus covering many cases in one stroke. At least two papers have appeared recently guided by a similar philosophy: \cite{Bo},  \cite{CT}.
\par
For individual multiplication or composition operators on the basic spaces such as the Hardy or Bergman spaces of the disk, the properties such as injectivity, surjectivity, or invertibility are relatively simple to understand; \textit{cf.\/} \cite{Ax}, \cite{V1}, \cite{CM}. However, for the weighted composition operators it is no longer so simple to describe the operators with such properties. In the recent papers \cite{G} and \cite{HLNS} the invertibility and Fredholm property of weighted composition operators were studied, in some special cases leading to identifying the spectrum of such a transformation. In \cite{G} the invertibility was characterized on the Hardy space $H^2$ and in \cite{HLNS} on more general spaces that satisfy a number of conditions. Our treatment is also axiomatic but our assumptions are much weaker, hence our results are far more general and some even work for spaces on arbitrary bounded domains. They also complement or generalize some of the results from \cite{Bo} where invertibility was considered even in general classes of analytic functions on the disk, not necessarily linear spaces. We explain below the main results and the organization of the paper.
\par
In Section~\ref{sec-prelim} we review a list of classical function spaces and examine some consequences of the basic axioms such as the boundedness of point evaluations. Among other results, we obtain a statement that might be of independent interest. Namely, it is well-known that every pointwise multiplier of a ``reasonable'' Banach space of analytic functions into itself should be a bounded analytic function and the converse is false in general. We find a condition (Theorem~\ref{thm-mult}) that describes in terms of a very simple property (called the Domination Property) all functional Banach spaces in which all bounded analytic functions are pointwise multipliers.
\par
In Section~\ref{sec-wco-charact} we consider functional Banach spaces of the unit disk that satisfy only five basic and very natural axioms, among them the boundedness of the shift operator (multiplication by the independent variable $z$). It is well-known that in the Hardy space $H^2$ the only operators that commute with the shift are the multiplication operators $M_F$ (by bounded analytic functions). It is also a classical result that all multiplicative operators on $H^2$, in the sense that $T(fg)=Tf\cdot Tg$ for all $f$, $g\in H^2$ with $fg\in H^2$ are precisely the composition operators $C_\vf$. For any space $X$ that satisfies our five basic axioms and a non-trivial bounded operator $T$ on $X$, Theorem~\ref{thm-wco-charact} extends both results in one stroke by giving several conditions on $T$, all of them equivalent to the property of being a weighted composition operator.
\par
We then use the above result to prove Theorem~\ref{thm-invert-comp-disk}, a characterization of invertible compact composition operators in all spaces that satisfy only our five basic axioms. This result generalizes similar results from \cite{G} and \cite{HLNS} and complements the results from \cite{Bo}. The results are accompanied by examples showing that if certain axioms are omitted, the results no longer hold. Since the axioms listed are so natural, we have limited the number of examples presented. It should be stressed that, even though a WCO is a product of two operators, $C_\vf$ and $M_F$ and often both of them are bounded, it is still possible for both of them to be unbounded while the product $T_{F,\vf}$ is bounded. What might come as a surprise is our Example~D where a functional Banach space of the disk and a bounded WCO are exhibited with the following properties: neither $C_\vf$ nor $M_F$ is bounded; however, the product $T_{F,\vf}$ is a bijective isometry of the space and an involution.
\par
In Section~\ref{sec-invert-5axioms} we consider functional Banach spaces on general bounded domains in the plane (without any connectivity assumptions whatsoever). Again, we require five axioms to be satisfied, some of them slightly different from those considered in  Section~\ref{sec-wco-charact}. These axioms are fulfilled in a number of relevant spaces of the disk. The main result of the section is Theorem~\ref{thm-wco-invert} which characterizes the invertible WCOs in arbitrary spaces that satisfy these axioms. The statement has the same flavor as Theorem~\ref{thm-invert-comp-disk} and the results from the recent papers cited but seems significantly more general than any of them. The proof relies in part on Theorem~\ref{thm-self-map} which may be of independent interest and which shows that analytic self-maps of bounded domains with certain boundary behavior must be onto.

%%%%%%%%%%%%%%%%%%%%%%%%%%%%%%%%%%%%%%%%%%%%%%%%%%%%%%%%%%%%%%%%
\section{Some preliminary results}
 \label{sec-prelim}
%%%%%%%%%%%%%%%%%%%%%%%%%%%%%%%%%%%%%%%%%%%%%%%%%%%%%%%%%%%%%%%%

%%%%%%%%%%%%%%%%%%%%%%%%%%%%%%%%%%%%%%%%%%%%%%%%%%%%%%%%%%%%%%%%
\subsection{Functional Banach spaces}
 \label{subsec-fbs}
%%%%%%%%%%%%%%%%%%%%%%%%%%%%%%%%%%%%%%%%%%%%%%%%%%%%%%%%%%%%%%%%
Denote by $\Omega$ a domain (\textit{i.e.,} an open and connected set) in the complex plane. We will consider general domains $\Omega$ which will sometimes be required to be bounded and sometimes will simply be the unit disk, denoted by $\D$.
\par
We will denote by $\ch (\Omega)$ the algebra of all analytic functions in $\Omega$. Let $X\subset\ch (\Omega)$ be a Banach space of analytic functions in $\Omega$. Abstract Banach spaces of analytic functions, assumed to satisfy only a handful of axioms, have been considered in the literature,  \textit{e.g.\/}, in \cite{BrS}, \cite{Sh1}, \cite{ADMV}, or \cite{Bo}. \par
Given a point $z$ in $\Omega$, we will denote by $\La_z$ the \textit{point evaluation functional\/} corresponding to $z$, defined by $\La_z(f)=f(z)$, for $f\in X$. Throughout the paper we assume that the point evaluation functionals are bounded in $X$. (This is the most common axiom in the literature.) Thus, they are uniformly bounded on each compact subset of $\Omega$; indeed, given a compact set $K\subset\Omega$, we have
$$
 \sup_{z\in K} |\La_z (f)| = \sup_{z\in K} |f(z)| < \infty
$$
for each $f\in X$, hence $\sup_{z\in K} \|\La_z\|<\infty$ by a direct application of the uniform boundedness principle. Our spaces will also often be assumed to contain the polynomials and we may also require that the polynomials be dense in the norm of $X$. Some additional axioms, satisfied by many specific spaces, will be assumed in different sections.
\par
%%%%%%%%%%%%%%%%%%%%%%%%%%%%%%%%%%%%%%%%%%%%%%%%%%%%%%%%%%%%%%%%
\subsection{Examples of function spaces}
 \label{subsec-spaces}
%%%%%%%%%%%%%%%%%%%%%%%%%%%%%%%%%%%%%%%%%%%%%%%%%%%%%%%%%%%%%%%%
We give here a partial list of classical Banach spaces of analytic functions that will be relevant for further discussion.
\begin{itemize}
 \item
The classical Hardy spaces $H^p$, $1\le p<\infty$, defined as the set of functions in $\ch (\D)$ such that
$$
 \|f\|_{H^p}=\sup_{0\le r<1} M_p(r,f) = \sup_{0\le r<1} \( \int_0^{2\pi} |f(r e^{i \theta})|^p \frac{d\theta}{2\pi}\)^{1/p} <\infty\,.
$$
(See \cite{D}, \cite{Ko}, among other sources.)
\item
The Bergman spaces $A^p$, $1\le p<\infty$ (\cite{DuS}, \cite{HKZ}), consisting of all functions in $\ch (\D)$ that are $p$-integrable with respect to the normalized Lebesgue area measure $dA(z)=\frac1{\pi} d x d y = \frac1{\pi} r dr d\theta$ ($z=x+iy=r e^{i\theta}\in\D$), equipped with the norm
$$
 \|f\|_{A^p} = \( \int_\D |f(z)|^p dA(z)\)^{1/p} <\infty\,.
$$
\item
The mixed-norm spaces $H(p,q,\a)$ (considered by Hardy and Littlewood and formally first introduced and studied by Flett \cite{Fl}, \cite{JVA}), consisting of all functions $f$ in $\ch(\D)$ for which
$$
 \|f\|_{p,q,\a} = \( \int_0^1 (1-r)^{\alpha q-1} M_p^q(r,f)  dr\)^{1/q} <\infty\,, \qquad 0<q<\infty\,,
$$
and
$$
 \|f\|_{p,\infty,\a} = \sup_{0\le r<1} (1-r)^{\alpha} M_p(r,f) <\infty\,.
$$
Here also $0<p\le\infty$ and $\alpha>0$. Of course, $H(p,p,1/p)=A^p$.
\item
The weighted Hilbert spaces $H^2_\beta$, introduced by Shields \cite{Shie} (see also \cite{CM}), consisting of all analytic functions in $\D$ with the Taylor series $f(z)=\sum_{n=0}^\infty a_n z^n$ in $\D$ and such that
$$
 \|f\|_{\beta}^2 = \sum_{n=0}^\infty \beta(n)^2 |a_n|^2
$$
(thus, generalizing the Hardy, Dirichlet, and Bergman spaces, with $\beta(n)^2=1$, $n$, and $(n+1)^{-1}$ respectively).
\item
The disk algebra $\ca=H^\infty\cap \cc(\overline{\D})$, equipped with the norm from the larger space $H^\infty$.
\item
The Bloch space $\cb$ (introduced by Anderson, Clunie and Pommerenke; see \cite{DuS}, \cite{HKZ}, \cite{Zh1}) of all $f$ in $\ch(\D)$ for which
$$
 \|f\|_\cb=|f(0)|+\sup_{z\in\D} (1-|z|^2) |f^\prime (z)|
$$
and its closed subspace $\cb_0$ called the little Bloch space (the closure of the polynomials in the above norm) whose functions satisfy
$$
 \lim_{|z|\to 1^-} (1-|z|^2) |f^\prime (z)| = 0\,.
$$
\item
The analytic (diagonal) Besov spaces $B^p$, $1<p<\infty$, of all functions analytic in $\D$ for which
$$
 \|f\|_{B^p}=|f(0)|+ \( (p-1) \int_{\D} (1-|z|^2)^{p-2} |f^\prime (z)|^p dA(z)\)^{1/p}<\infty\,,
$$
all of them being invariant under compositions with disk automorphisms, contained in $\cb_0$ and containing the minimal conformally invariant subspace $B^1$ of all functions for which $f^{\prime\prime}\in A^1$. It is well known that $B^p\subset \cb_0 \subset\cb$.See \cite[Chapter~5]{Zh1}.
\item
The weighted growth spaces $H_v^\infty$ defined by the condition
$$
 \|f\|_v = \sup_{z\in\D}\,v(z)\,|f(z)| < \infty\,,
$$
where the weight $v$ is a strictly positive function, some typical examples being $v(z)=(1-|z|)^\alpha$ and $v(z)=e^{-|z|^2}$. See, \textit{e.g.\/}, \cite{BDLT} or \cite{CHD1}.
\item
The Bargmann-Fock space $F^2$ of entire functions \cite{Zh2} which are square integrable in the complex plane $\C$ with respect to the Gaussian measure:
$$
 \|g\|_F^2 = \frac{1}{\pi} \int_\C |g(z)|^2 e^{-|z|^2} dA(z)<\infty\,. $$
\end{itemize}

%%%%%%%%%%%%%%%%%%%%%%%%%%%%%%%%%%%%%%%%%%%%%%%%%%%%%%%%%%%%%%%%
\subsection{Some consequences of a basic axiom in the disk}
 \label{subsec-conseq-axiom}
%%%%%%%%%%%%%%%%%%%%%%%%%%%%%%%%%%%%%%%%%%%%%%%%%%%%%%%%%%%%%%%%
Let $X\subset\ch(\D)$ be a functional Banach space on the disk on which the point evaluations are bounded. If $f\in X$ and $f(z)= \sum_{n=0}^\infty a_n z^n$ in $\D$ then the functional $\La_0$, given by $\La_0(f)=a_0=f(0)$,  is bounded. This observation extends to the remaining coefficient functionals.
\par
%%%%%%%%%%
\begin{prop} \label{prop-coeff-est}
Let $X\subset\ch(\D)$ be a Banach space which contains the polynomials and on which the point evaluations are bounded. The following assertions hold.
\par
(a) All coefficient functionals $\La_n (f) =a_n$ are bounded. More precisely, for every $n\in\N$ and $r\in (0,1)$ there exists a constant $M_r>0$ such that \ $|a_n| \le \dfrac{M_r}{r^n} \|f\|_X$ \ for all $f\in X$.
\par
(b) $\limsup_{n\to\infty} \sqrt[n]{\|z^n\|_X}\ge 1$.
\end{prop}
%%%%%%%%%%
\begin{proof}
Recall that the point evaluation functionals on $X$ are uniformly bounded on compact subsets of the disk. Thus,
$$
 \max_{|w|=r} |f(w)| \le M_r \|f\|_X
$$
for some fixed $M_r>0$ and all $f\in X$. Let $f\in X$ with $f(z)= \sum_{n=0}^\infty a_n z^n$. The Cauchy integral formula yields
$$
 |a_n| \le \frac{\max_{|w|=r} |f(w)|}{r^n} \le	\frac{M_r}{r^n} \|f\|_X \,.
$$
\par
(b) Applying the conclusion of part (a) to the function $f(z)=z^n$, we get $1\le \dfrac{M_r}{r^n} \|f\|_X$ for arbitrary but fixed $r$, $0<r<1$. Hence
$$
 r = \lim_{n\to\infty} \frac{r}{\sqrt[n]{M_r}} \le \limsup_{n\to\infty} \sqrt[n]{\|z^n\|_X} \,.
$$
Since this holds for arbitrary $r\in (0,1)$, the statement follows.
\end{proof}
%%%%%%%%%%
\par
In the next section a relevant condition in some statements will be the assumption that $\limsup_{n\to\infty} \sqrt[n]{\|z^n\|_X}\le 1$ and, thus, actually $\limsup_{n\to\infty} \sqrt[n]{\|z^n\|_X}=1$. It is not difficult to see that in most spaces of interest to us the above $\limsup$ is actually equal to one. This is readily verified for the Hardy, Bergman, and mixed-norm spaces but also in the Bloch, analytic Besov spaces, and weighted Banach spaces $H_v^\infty$.
\par
The next example may not seem so natural but will be relevant for further discussion.
%%%%%%%%%%
\begin{ex} \label{ex-restr-fock}
Consider the space $X$ whose elements $f|_\D$ are the restrictions to the disk of the functions $f$ in the Bargmann-Fock space $F^2$. Clearly, if $f|_\D\in X$, its extension $f$ to the entire plane is unique by the principle of isolated zeros, hence there is a one-to-one correspondence between the members of $X$ and the functions in $F^2$. If $X$ is equipped by the norm of the extension to the plane: $\|f|_\D\|_X = \|f\|_{F^2}$, it is clear that it is a Banach space (actually, Hilbert) of analytic functions in the unit disk. It is also well-known that the point evaluations (at all points in the plane) are bounded on $F^2$ \cite[Theorem~2.7]{Zh2} and the polynomials are dense  \cite[Proposition~2.9]{Zh2}, hence the point evaluations (at the points of $\,\D$) are bounded on $X$ and the polynomials are dense in $X$. So, the space $X$ defined in this fashion is one more in our collection of spaces; however, it does not satisfy the condition $\limsup_{n\to\infty} \sqrt[n]{\|z^n|_\D\|_X}=1$. In fact, one easily calculates that
$$
 \|z^n|_\D\|_X^2 = \|z^n\|_{F^2}^2 =2 \int_0^\infty r^{2n+1} e^{-r^2} dr = \Gamma(n+1) = n!
$$
and Stirling's formula shows that $\lim_{n\to\infty}\sqrt[n]{\|z^n\|_X} =\infty$.
\end{ex}
%%%%%%%%%%
\par

%%%%%%%%%%%%%%%%%%%%%%%%%%%%%%%%%%%%%%%%%%%%%%%%%%%%%%%%%%%%%%%%
\subsection{Pointwise multipliers and the domination property}
 \label{subsec-ptwse-mult}
%%%%%%%%%%%%%%%%%%%%%%%%%%%%%%%%%%%%%%%%%%%%%%%%%%%%%%%%%%%%%%%%
A function $F$ analytic in a planar domain $\Omega$ is said to be a \textit{pointwise multiplier\/} of a Banach space of analytic functions $X$ into itself if $F f\in X$ for every $f\in X$. For any such $F$ we can define the	 \textit{pointwise multiplication operator\/} $M_F$ in a natural way:
$$
 M_F : X \to X\,, \quad M_F f = F f\,, \qquad f\in X\,.
$$
Under the assumption that each point-evaluation functional $\La_\zeta(f)=f(\zeta)$ is bounded on $X$, a standard normal family argument shows that $M_F$ has closed graph and is, thus, a bounded operator. Pointwise multipliers on various spaces of analytic functions were examined in a number of papers, \textit{e.g.\/}, in \cite{At}, \cite{Ax}, \cite{BrS}, \cite{V1, V2}, \cite{Z}.
\par
Let us denote by $M_F$ the pointwise multiplication operator with symbol $F$ and by $\mathcal{M}(X)$ the space of all (bounded) multipliers from $X$ into itself, a closed subspace of the space of bounded operators on $X$. It is a simple consequence of the boundedness of point evaluations that $\mathcal{M}(X)\subseteq H^\infty(\Omega)$, the space of all bounded analytic functios in $\Omega$ and	 $\|F\|_\infty\le \|M_F\|$. See, for example, \cite[Lemma~1.1]{ADMV} for more details.
\par
It is known that, in general, $H^\infty(\Omega)\neq\cm(X)$; for example, in the Bloch space $\cb$ there exist bounded non-multipliers of the space; see \cite{BrS}. Thus, a most natural question comes to mind: for which ``reasonable'' Banach spaces $X$ of analytic functions is $H^\infty(\Omega) = \cm(X)$ true? We have not been able to find an answer in the literature and we give here a very simple answer in terms of what we call the domination property.
\par\medskip
We will say that a Banach space $X$ of analytic functions has the \textit{Domination Property\/} (DP) if whenever $f\in\ch(\Omega)$,  $g\in X$ and $|f(z)|\le |g(z)|$ holds for all $z\in\Omega$, then $f\in X$.  \par\medskip
It is readily checked that all Hardy, Bergman, mixed-norm spaces $H(p,q,\a)$, and weighted Banach spaces $H_v^\infty$ have this property. We will now see that (DP) is actually equivalent to its stronger form and characterizes all spaces in which the multipliers coincide with bounded analytic functions.
\par
%%%%%%%%%%
\begin{thm} \label{thm-mult}
Let $\Omega$ be a planar domain and $X\subset\ch(\Omega)$ a functional Banach space in which the point evaluations are bounded. Then the following conditions are equivalent:
\begin{itemize}
 \item[(a)]
$H^\infty(\Omega) \subset \cm(X)$; that is, $H^\infty(\Omega)=\cm(X)$.
 \item[(b)]
$X$ has (DP).
 \item[(c)]
There is a universal constant $C>0$ such that if $f\in\ch(\Omega)$,  $g\in X$ and $|f(z)|\le |g(z)|$ holds for all $z\in\Omega$, then $f\in X$ and $\|f\|_X\le C \|g\|_X$.
\end{itemize}
Moreover, the least constant possible in the inequality that defines the property (DP) is $C=\|J\|$, where $J$ is the correspondence operator $J\,\colon\,H^\infty\to\cm(X)$, $J(F)=M_F$.
\par
In the special case when $\Omega=\D$, both conditions above are equivalent to the following:
\begin{itemize}
 \item[(d)]
The set $BP$ of all Blaschke products satisfies $BP\subset\cm(X)$ and $\sup_{B\in BP}\|M_B\|_{X\to X}<\infty$;
\end{itemize}
\end{thm}
%%%%%%%%%%
\begin{proof}
\fbox{(c)$\Rightarrow$(b)} \ Trivial.
\par\medskip
\fbox{(b)$\Rightarrow$(a)} \ Suppose $X$ has (DP). Let $F\in H^\infty(\Omega)$ and let $f\in X$ be arbitrary. Then the function $\|F\|_\infty f\in X$ and $|F(z) f(z)|\le \|F\|_\infty |f(z)|$
holds for all $z\in\Omega$, hence by (DP) we have $F f\in X$. This shows that $F\in\cm(X)$.
\par\medskip
\fbox{(a)$\Rightarrow$(c)} \ Now assume that $H^\infty(\Omega) \subset \cm(X)$ and let us show that (c) holds.
\par
First of all, observe that the correspondence operator $J\,\colon\,H^\infty\to\cm(X)$, given by $J(F)=M_F$, is bounded. To verify this, it suffices to see that it has closed graph by an  application of a normal families argument in the usual way: suppose that $F_n\to F$ in $H^\infty(\Omega)$ and $J(F_n)=M_{F_n}\to T$ in $\cb(X)$. In order to show that $T=M_F$, we note that
$$
 \|F_n f-Tf\|_X\le \|M_{F_n}-T\| \cdot \|f\|_X \to 0\,, \qquad n\to \infty\,.
$$
In view of the boundedness of point evaluations, it follows that
$F_n\to F$ pointwise (actually, uniformly on compact subsets) and also $F_n f\to Tf$ pointwise, hence $F_n f\to F f$ pointwise and therefore $Tf=F f$. This shows that $T=M_F$.
\par
Knowing that $J$ is a bounded operator, there exists a universal constant $C>0$ such that for each $F\in H^\infty(\Omega)$ we have $\|M_F\|\le C \|F\|_\infty$. Thus, whenever $\|F\|_\infty\le 1$, we have that $\|M_F\|\le C$ for some fixed constant $C$.
\par
Let $f\in\ch(\Omega)$ and $g\in X$ be such that $|f(z)|\le |g(z)|$ holds for all $z\in\D$. The trivial case $g\equiv 0$ yields $f\equiv 0$ which clearly presents no problem. When $g$ is not identically zero in $\Omega$, every zero of $g$ is a zero of $f$ (of at least the same order) and is isolated so one easily extends the function $F=f/g$ to be analytic in the whole domain $\Omega$. This function $F\in H^\infty(\Omega)$ and $\|F\|_\infty\le 1$. By the above observation based on the closed graph theorem, there is a fixed constant $C$ (independent of $F$) such that $\|M_F\|\le C$. Hence, we obtain
$$
 \|f\|_X = \|F g\|_X \le \|M_F\| \|g\|_X \le C \|g\|_X\,.
$$
This shows that (c) holds. It is also clear from the proof that the smallest possible value of $C$ is precisely $\|J\|$.
\par\medskip
In the case when $\Omega=\D$, it is readily seen that any of the assumptions (a), (b), or (c) implies (d). The converse follows from Marshall's theorem on the density of the convex hull of Blaschke products in the unit ball of $H^\infty$ \cite[Chapter~VII~B]{Ko}.
\end{proof}
%%%%%%%%%%
%\par
%The operator $J$ above does not have to be isometric. The question %arises as to whether its norm could be one; that is, whether %$\|M_F\|=\|F\|_\infty$. It turns out that the answer is negative, as %the following example shows.
\par
%%%%%%%%%%
%\begin{ex} \label{ex-isom}
%Consider the space $X$ given by\ldots
%\end{ex}
%%%%%%%%%%
%%%%%%%%%%%%%%%%%%%%%%%%%%%%%%%%%%%%%%%%%%%%%%%%%%%%%%%%%%%%%%%%
\subsection{Composition operators}
 \label{subsec-compos-ops}
%%%%%%%%%%%%%%%%%%%%%%%%%%%%%%%%%%%%%%%%%%%%%%%%%%%%%%%%%%%%%%%%
Given an analytic function $\vf\in\ch (\Omega)$ such that $\vf(\Omega)\subset \Omega$, we will say that $\vf$ is an \textit{analytic self-map\/} of $\Omega$. It is natural to define the \textit{composition operator\/} $C_\vf$ with \textit{symbol\/} $\vf$ by the formula $C_\vf f=f\circ\vf$. In the case when case $\Omega=\D$, for a number of known spaces of the disk, such as the Hardy, Bergman, $H(p,q,\a)$, Bloch spaces, every analytic self-map of the disk defines a composition operator of the space into itself (which is automatically bounded by an application of the closed graph theorem). For a good exposition of the theory of composition operators on Hardy spaces and its rich interplay with geometric function theory and iterations of self-maps of the disk, we refer the reader to the monographs \cite{Sh1} and \cite{CM}.
\par
In $B^p$ spaces, only the maps $\vf$ that satisfy a certain Carleson measure condition define bounded composition operators. However, all disk automorphisms do have this property and even preserve the norm; in other words, these spaces are conformally invariant; see \cite{Zh1}.  In the weighted $H^\infty_v$ spaces, again, not all self-maps induce bounded composition operators \cite{BDLT}. In the Fock space, only certain linear maps are symbols of bounded composition operators \cite{CMS}.
\par
It is worth remarking that it does not seem easy to find a natural and satisfactory analogue of Theorem~\ref{thm-mult} for composition operators. At least we have not been able to identify a property that would play a role analogous to that of (DP) in this context.

%%%%%%%%%%%%%%%%%%%%%%%%%%%%%%%%%%%%%%%%%%%%%%%%%%%%%%%%%%%%%%%%
\section{Weighted composition operators on the disk}
 \label{sec-wco-charact}
%%%%%%%%%%%%%%%%%%%%%%%%%%%%%%%%%%%%%%%%%%%%%%%%%%%%%%%%%%%%%%%%

In this section we only consider spaces of analytic functions in the disk that satisfy a fixed set of five natural axioms. We will prove two results: one characterizing all weighted composition operators among the bounded operators on such a space, and another, characterizing the invertible WCOs among the bounded ones.
%%%%%%%%%%%%%%%%%%%%%%%%%%%%%%%%%%%%%%%%%%%%%%%%%%%%%%%%%%%%%%%%
\subsection{A characterization of weighted composition operators on spaces of the disk}
 \label{subsec-wco-charact}
%%%%%%%%%%%%%%%%%%%%%%%%%%%%%%%%%%%%%%%%%%%%%%%%%%%%%%%%%%%%%%%%
\par
The unilateral shift operator $S$, defined by  $Sf(z)=z f(z)$, for $z\in\D$ and $f\in X$, makes sense and is often bounded on many functional Banach spaces $X\subset \ch(\D)$. Clearly, $S=M_z$.
\par
Throughout the whole section, we will consider Banach spaces $X\subset \ch(\D)$ that satisfy the following set of axioms:
\par
\begin{description}
\item[Ax1]
 All point evaluation functionals $\La_z$ are bounded on $X$.
\item[Ax2]
 The set of all algebraic polynomials of $z$ is contained in $X$ and dense in it (in $\|\cdot\|_X$).
\item[Ax3]
 The shift operator $S=M_z$ is bounded on $X$.
\item[Ax4]
 $\limsup_{n\to\infty} \|z^n\|^{1/n}= 1$.
\item[Ax5]
 Every disk automorphism induces a bounded composition operators on $X$.
\end{description}
\par
Each one of these axioms appears in a most natural way in one context or another. The disk algebra $\ca$, the Hardy spaces $H^p$ and the Bergman spaces $A^p$, $1\le p<\infty$, satisfy all five axioms. More generally, the spaces $H(p,q,\alpha)$ verify them when $1\le p\le\infty$ and $1\le q<\infty$. Also, the little Bloch space $\cb_0$ and analytic Besov spaces $B^p$, $1<p<\infty$, satisfy all of these axioms (see \cite[Chapter~5]{Zh1}, \cite{HW}, or \cite{BV}).
\par
The weighted spaces $H^2 (\beta)$ satisfy Axiom~\textbf{[Ax1]} by \cite[Theorem~2.10]{CM} and clearly also Axiom~\textbf{[Ax2]} (since the Taylor polynomial of each function converges to the same function in the norm). They satisfy Axiom~\textbf{[Ax3]} if and only if $\sup_n \beta(n+1)/\beta(n)<\infty$ by \cite[Proposition~2.7]{CM} and Axiom~\textbf{[Ax4]} if and only if $\limsup_n \beta(n)^{1/n}=1$ by a trivial calculation; a simple example of a coefficient weight that satisfies both is $\beta (n)=n^\alpha$, $\alpha\ge 0$ or $\beta (n)=(n+1)^\alpha$, $\alpha< 0$. The weighted space $H^2 (\beta)$ with $\beta (n)=2^n$ fails to satisfy Axiom~\textbf{[Ax5]}, as \cite[Exercise~3.1.5]{CM} shows, while the Dirichlet space $B^2=H^2 (\beta)$ with $\beta (n)=\sqrt{n}$ does.
\par
The non-separable spaces such as $\cb$, $H^\infty$, and also $H^\infty_v$ for certain weights fail to satisfy Axiom~\textbf{[Ax2]}.
\par\medskip
Theorem~\ref{thm-wco-charact}, to be proved in this subsection, provides a generalization of the well-known fact that in Hardy spaces or Bergman spaces the only operators that commute with the shift are the pointwise multiplication operators. It also generalizes a well-known characterization of composition operators on $H^2$ which states that they are the only multiplicative operators on this space, in the sense that $T(fg)= Tf\cdot Tg$ for all $f$, $g\in H^2$ such that also $f g\in H^2$; see \cite[Corollary~5.1.14, p.~170]{MAR}.
\par
Before stating and proving Theorem~\ref{thm-wco-charact} we will need a simple auxiliary statement, easy to prove and valid for arbitrary planar domains  \cite[Lemma~11, p.~57]{DuS}; however, we will state it only for the disk.
%%%%%%%%%%
\begin{lem} \label{lem-merom}
Let $g\in \ch(\D)$, $g\not\equiv 0$, let $H$ be meromorphic in $\D$, and assume that $g H^n\in\ch(\D)$ for all $n\ge 1$. Then $H\in \ch(\D)$.
\end{lem}
%%%%%%%%%%
\par
It will become obvious from the proof that the extra assumption on the norms of the monomials, together with an assumption on kernel of $T$, forces the function $\vf=Tz/(T1)$ to become an analytic self-map of the disk, which is the key to the validity of the result. It should be noted that Axiom~\textbf{[Ax5]} is not required in the statement and also that we do not exclude the trivial case of constant self-maps of the disk: $\vf\equiv c$, $|c|<1$.
%%%%%%%%%%
\begin{thm} \label{thm-wco-charact}
Let $X\subset\ch(\D)$ be a functional Banach space in which the axioms \textbf{[Ax1] - [Ax4]} are fulfilled and let $T$ be a continuous operator on $X$ such that the function $z-\la \not\in \mathrm{ker} T$ for any unimodular number $\la$. Then the following conditions are equivalent:
\begin{itemize}
 \item[(a)]
$T$ is a weighted composition operator;
 \item[(b)]
There exists $\,\vf\in\ch(\D)$ such that $\,\vf(\D)\subset\D$ and $M_\vf T = T S$;
\item[(c)]
There exists $\,\vf$, meromorphic in $\D$, such that $M_\vf T = T S$;
 \item[(d)]
There exists $\,\vf$, meromorphic in $\D$, such that $M_{\vf^n} T = T S^n$ for all integers $n\ge 0$;
 \item[(e)]
There exists $\,\vf$, meromorphic in $\D$, such that $\vf^n \cdot T1 = T (z^n)$ for all integers $n\ge 0$;
\item[(f)]
There exists $\,\vf\in\ch(\D)$ such that $\,\vf(\D)\subset\D$ and $\vf^n \cdot T1 = T (z^n)$  for all integers $n\ge 0$;
\item[(g)]
$T1\cdot T(fg) = Tf\cdot Tg$ holds for all functions $f$, $g\in X$ for which $fg\in X$ as well.
\end{itemize}
If any of the conditions (a)--(g) is fulfilled, then $\vf=Tz/(T1)$ is the composition symbol of $T$.
\end{thm}
%%%%%%%%%%
\begin{proof} It suffices to prove the chain of implications (a) $\Rightarrow$ (b) $\Rightarrow$ (c) $\Rightarrow$ (d) $\Rightarrow$ (e) $\Rightarrow$ (f) $\Rightarrow$ (a) and then also (a) $\Rightarrow$ (g) $\Rightarrow$ (e).
\par\medskip
\fbox{(a) $\Rightarrow$ (b)} \ Let $T=M_F C_\vf$. Then for any $f$ in $X$ we have $M_\vf T f = F \vf (f\circ \vf) = T S f$, hence $M_\vf T = T S$ holds. (Note that this step involves using \textbf{[Ax3]}.)
\par\medskip
\fbox{(b) $\Rightarrow$ (c)} \ This implication is trivial.
\par\medskip
\fbox{(c) $\Rightarrow$ (d)} \ The statement is obvious for $n=0$ and we also have it for $n=1$ by (c). Now $M_\vf T = T S$ implies $T S^2 = M_\vf T S = M_\vf M_\vf T = M_{\vf^2} T$. Proceeding by induction, we obtain $T S^n = M_{\vf^n} T$ for all $n\ge 0$.
\par\medskip
\fbox{(d) $\Rightarrow$ (e)} \ Follows trivially by applying both sides of the equality in (d) to the constant function one.
\par\medskip
\fbox{(e) $\Rightarrow$ (f)} \ First, we have to distinguish between two possibilities: $T1\equiv 0$ and $T1\not\equiv 0$. The possibility of $T1\equiv 0$ is easily excluded: in this case, (e) implies that also $Tz\equiv 0$. But this possibility is excluded by our assumption on the kernel of $T$.
\par
Now, knowing that $T1\not\equiv 0$, it follows from the assumption (e)  for  $n=1$ that $\vf=Tz/(T1)$, a meromorphic function. Since $T(z^n)= T1 \cdot \vf^n$ holds for all integers $n\ge 1$, it follows from Lemma~\ref{lem-merom} that actually $\vf\in\ch(\D)$.
\par
It is only left to  show that $|\vf(\z)|<1$ for all $\z\in\D$. Since $T1\cdot \vf^n\in X$ for all $n\ge 1$, by the boundedness of the point-evaluation functionals for an arbitrary but fixed $\z$ in $\D$ (by virtue of \textbf{[Ax1]}) we get
$$
 |(T1\cdot \vf^n) (\z)|\le \|\La_\z\|\cdot \|T1\cdot \vf^n\| = \|\La_\z\| \cdot \|T (z^n)\| \le \|\La_\z\|\cdot \|T\| \cdot \|z^n\|\,.
$$
Taking the $n$-th roots of both sides and then $\limsup_{n\to\infty}$, using condition \textbf{[Ax4]} we first get that  $|\vf(\z)|\le 1$ for all $\z\in\D$.
\par
If $|\vf(\z)|=1$ for some $\z\in\D$ then by the maximum modulus principle $\vf$ is identically constant, say $\vf(z)\equiv c$, $|c|=1$. But then $Tz= c T1$ and thus $z-c \in \mathrm{ker} T$, contrary to our assumption. Thus, $\vf$ is an analytic function in $\D$ which maps $\D$ into itself.
\par\medskip
\fbox{(f) $\Rightarrow$ (a)} \  Suppose that $\vf^n \cdot T1 = T (z^n)$ for some analytic self-map $\vf$ of $\D$ and for all integers $n\ge 0$. We need to show that $T$ is a weighted composition operator, namely, $Tf=F\cdot (f\circ\vf)$ for this same mapping $\vf$ and $F=T1$.
\par
By assumption, $T(z^n)=T1\cdot\vf^n=F\cdot \vf^n$ for all integers $n\ge 0$, hence by linearity $Tp=F\cdot (p\circ\vf)$ for every polynomial $p$. Now let $f$ be an arbitrary function in $X$ and $(p_n)_n$ a sequence of polynomials convergent to $f$ in the norm of $X$ (which exists by \textbf{[Ax2]}). Then, by continuity of $T$, we have $Tp_n\to Tf$ in $X$ and since $Tp_n=F\cdot (p_n\circ\vf)$, it follows that $F\cdot (p_n\circ\vf)\to Tf$ in $X$, and hence pointwise as well. That is, $Tf(z)=F(z)\cdot (f\circ\vf)(z)$ at every point $z$ in $\D$.
\par\medskip
\fbox{(a) $\Rightarrow$ (g)} \  Let $Tf=F (f\circ \vf)$. It is readily  checked that (g) holds as both sides of the equality $T1\cdot T(fg) = Tf\cdot Tg$ are equal to $F^2 (f\circ \vf) (g\circ \vf)$.
\par\medskip
\fbox{(g) $\Rightarrow$ (e)} \  Assume that $T1\cdot T(fg) = Tf\cdot Tg$ holds whenever $f$, $g$, $fg\in X$. It is clear that the identity in (e) holds trivially for $n=0$ and that it holds for $n=1$ for $\vf=Tz/(T1)$, clearly a meromorphic function in $\D$ since $Tz$, $T1 \in X$. We now prove that $\vf^n \cdot T1 = T (z^n)$ by induction on $n\ge 1$. Assuming that $T(z^n)= \vf^n \cdot T1 = \frac{(Tz)^n}{(T1)^{n-1}}$ for some $n\ge 1$, we get from (g) with $f(z)=z^n$ and $g(z)=z$ and from the inductive hypothesis that
$$
 T(z^{n+1})= \frac{T(z^n) T(z)}{T1} = \frac{(Tz)^{n+1}}{(T1)^n} = \vf^{n+1} \cdot T1\,,
$$
which completes the inductive proof.
\end{proof}
%%%%%%%%%%
\par
An inspection of the proof shows that we have effectively used all assumptions on $X$ and on $T$. It should be clear that they are all really needed but we still include some examples to illustrate this.
\par
The following example shows that the statement is no longer true if we omit the assumption on $\mathrm{ker} T$.
\par
%%%%%%%%%%
\begin{ex} \label{ex-rank}
Let $X=\ca=H^\infty(\D)\cap C(\overline{\D})$ be the disk algebra, that is, the space of all functions analytic in the unit disk and continuous up to the boundary. It is easy to check that this space satisfies \textbf{[Ax1]}--\textbf{[Ax5]}. Let $T$ be the rank one operator on $\ca$ given by $Tf=f(1) F_0$, where $F_0\in\ca$ is a non-zero function. Then clearly $Tz=F_0=T1$, hence $z-1\in \mathrm{ker} T$, so this is the only assumption in Theorem~\ref{thm-wco-charact} that is not satisfied. We will now check that in this situation the result no longer holds.
\par
Clearly, $M_\vf T = T S$ is satisfied for the function $\,\vf\equiv 1$ but there does not exist $\,\vf\in \ch(\D)$ with $\,\vf(\D)\subset\D$ such that $\vf^n \cdot T1 = T (z^n)$ for all integers $n\ge 0$ (as any such $\vf$ would have to be identically one). In other words, although the conditions (b)-(e) are satisfied, (f) is not.
\par
One can also check directly that $T$ is not a weighted composition operator. For example, substituting the constant function one into $f(1) F_0 = T_{F,\vf}f = F (f\circ\vf)$ yields that $F\equiv F_0$ and a further substitution $f(z)=z$ into $f(1) F_0 = F_0 (f\circ\vf)$ quickly leads to a contradiction.
\par
\end{ex}
%%%%%%%%%%
\par
The following example shows that the result is false if the assumption \textbf{[Ax4]} is not fulfilled. In this case, the remaining axioms \textbf{[Ax1]} - \textbf{[Ax3]} are valid in $X$ and, in spite of the fact that the representation $Tf=F\cdot (f\circ\vf)$ in view of (d) holds on a dense subset of $X$, namely for all polynomials, $T$ cannot be represented in this fashion on the whole space.
\par
%%%%%%%%%%
\begin{ex} \label{ex-non-WCO}
Let $X$ be the space from Example~\ref{ex-restr-fock} formed by restrictions of functions in the Fock space to the disk and let $T$ be the linear operator given by
$$
 Tf(z) = f\(\frac{z}{2}+1\)\,, \quad f\in X, \quad z\in\D\,.
$$
As some very simple integral estimates show, this is a bounded operator on $X$. (A more general result characterizing the symbols of all bounded composition operators on the Fock space can be found in \cite{CMS}.)
\par
It is trivial to check that $T$ satisfies condition (c) of Theorem~\ref{thm-wco-charact} with $\vf(z)=Tz/(T1)=\frac{z}{2}+1$. However, $\vf$ is not a self-map of $\,\D$ so other conditions do not hold. Thus, even though the formal relation $T=M_1 C_\vf= C_\vf$ holds, $T$ is not a weighted composition operator on the disk in the sense considered here -in view of the fact that $\vf$ does not map $\D$ into itself.
\par
It is important to note that there does not exist any other representation of $T$ as a WCO with an admissible pair of symbols. Indeed, if we had $T=C_\vf=M_F C_\psi$ for some $F\in\ch(\D)$ and some analytic self-map $\,\psi$ of $\,\D$, then taking first $f\equiv 1$ and then $f(z)=z$ in $f\circ\vf=F\cdot (f\circ\psi)$, this would imply  $\,F\equiv 1$ and $\psi=\vf$ respectively.
\end{ex}
%%%%%%%%%%
\par
Our next example shows that the assumption \textbf{[Ax2]} is also essential. Indeed, if $X$ does not contain the constant functions (even if the polynomials lacking the constant term are dense in the space), the conclusion of the result fails.
\par
%%%%%%%%%%
\begin{ex} \label{ex-isom}
Let $X=A^2_0=\{f\in A^2\,\colon\,f(0)=0\}$, the subspace of the Bergman space $A^2$ of codimension one equipped with inherited norm. It is clear that this is a Hilbert space and \textbf{[Ax1]}, \textbf{[Ax3]}, and \textbf{[Ax4]} are satisfied. Also, $X$ contains all polynomials that vanish at the origin and it is easy to see that they are actually dense in $X$. However, the constant function one does not belong to $X$.
\par
If we define the operator $T$ by $Tf(z)=f^\prime(0) z$, it is clear that this is a linear rank one operator from $X$ into itself and satisfies the relation $M_\vf T = TS = 0$ for the constant function $\vf\equiv 0$ (which trivially maps $\D$ into itself). However, the operator $T$ defined in this fashion cannot be written as a WCO in any way. In fact, if we had $F(z)\cdot f(\vf(z))=f^\prime(0) z$ for some $F$ analytic in $\D$ and some analytic mapping $\vf$ of $\D$ into itself, after substituting $f(z) \equiv z$ we would get $F(z) \vf(z)\equiv z$ while substituting $f(z) \equiv z^2$ would yield $F \vf^2\equiv 0$, hence either $F\equiv 0$ or $\vf \equiv 0$, immediately leading to a contradiction.
\end{ex}
%%%%%%%%%%
\medskip

%%%%%%%%%%%%%%%%%%%%%%%%%%%%%%%%%%%%%%%%%%%%%%%%%%%%%%%%%%%%%%%%
\subsection{Invertibility of WCOs in spaces on the disk}
 \label{subsec-bourdon}
%%%%%%%%%%%%%%%%%%%%%%%%%%%%%%%%%%%%%%%%%%%%%%%%%%%%%%%%%%%%%%%%
\par
We already know that in certain ``natural'' spaces the inverse of an invertible composition operator is again a composition operator though this does not hold in general; \textit{cf.\/} \cite[Exercises 2.1.14, 3.1.6]{CM}. We now address the following question: when is a weighted composition operator invertible? A rule of thumb should suggest that the composition symbol should be an automorphism of the domain and the multiplication symbol should be invertible (bounded from above and below) or, alternatively, self-multipliers of the space.
\par
Statements of this exact type have already appeared in the recent literature. Gunatillake \cite{G} studied the spectrum of weighted composition operators with automorphic symbols acting on the Hardy space $H^2$. As a motivation for this, in the same paper he also showed that such an operator is invertible on $H^2$ if and only if the composition symbol is an automorphism and the multiplication symbol is bounded from above and from below.
\par
Two generalizations were obtained subsequently: Bourdon \cite{Bo}  obtained a version of the last statement in the context of sets of analytic functions in the disk (with no linear structure at all) that satisfy certain axioms. He then applied his findings to certain specific spaces of the disk. Hyvärinen-Lindström-Nieminen-Saukko \cite{HLNS} obtained a generalization to other spaces of analytic functions on the disk with certain growth control.
\par
In this subsection we consider a different set of axioms and prove similar invertibility results which complement the earlier results. If compared with Bourdon's result \cite[Theorem~2.2]{Bo} which applies in a very general context of sets of functions without linear structure but with certain boundary properties, our result below gives an impression of being less general. However, we have neither been able to prove rigorously that our axioms \textbf{[Ax1] - [Ax4]} imply Bourdon's from \cite[Theorem~2.2]{Bo} nor to give an example of a space that satisfies our axioms but does not satisfy his. The main interest of the theorem may reside in the method of proof.
%%%%%%%%%%
\begin{thm} \label{thm-invert-comp-disk}
Let $X\subset\ch(\D)$ be a functional Banach space in which the axioms \textbf{[Ax1] - [Ax4]} are satisfied, as in Theorem~\ref{thm-wco-charact}, and suppose that a weighted composition operator $T_{F,\vf}$ is bounded in $X$.
\par
(a) If $T_{F,\vf}$ is invertible in $X$ then its composition symbol $\vf$ is an automorphism of \ $\D$, the multiplication symbol $\,F$ does not vanish in the disk, and the inverse operator $T_{F,\vf}^{-1}$ is another weighted composition operator $T_{G,\psi}$, whose symbols are:
\begin{equation}
 G=\frac{1}{F\circ \vf^{-1}}, \qquad \psi=\vf^{-1}\,.
 \label{inv-symb}
\end{equation}
\par
(b) Assuming that Axiom~\textbf{[Ax5]} also holds, we have the following characterization.
\par
The weighted composition operator $T_{F,\vf}$ is invertible on $X$ if and only if its composition symbol $\vf$ is an automorphism of \ $\D$, the multiplication symbol $\,F$ does not vanish in the disk, and $1/F\in\cm (X)$. If this is the case, then $F$ is also a self-multiplier of $X$ and the inverse operator is $T_{G,\psi}$, with the symbols given by \eqref{inv-symb}.
\end{thm}
%%%%%%%%%%
\par
Note that the condition $F\in\cm(X)$ is not needed in the proof of (b), hence it is not listed among the hypotheses. It actually follows easily from the remaining assumptions. This may seem paradoxical but can be explained by the fact that we are assuming from the start that $T_{F,\vf}$ acts boundedly so in certain multiplications the boundedness of $M_F$ is not really required.
%%%%%%%%%%
\begin{proof}
(a) To simplify the notation, write $T$ for our operator $T_{F,\vf}$ and $U$ for its inverse. Then $TU=UT=I$, where $I$ is the identity operator on $X$.
\par
We first make sure that the possibility $Uz=\la\cdot U1$ (for some fixed $\la$ with $|\la|=1$) is ruled out: indeed, if this happens then $z= TUz = \la TU1 = \la$ for all $z\in\D$, which is absurd. This shows the hypotheses of Theorem~\ref{thm-wco-charact} are all satisfied so we can apply the statement.
\par
We want to use part (e) of Theorem~\ref{thm-wco-charact} to conclude that $U$ is also a WCO. To this end, we ought to show that
\begin{equation}
 U(z^n) = U1 \cdot \(\frac{Uz}{U1}\)^n\,, \quad n\ge 2\,.
 \label{U-T}
\end{equation}
First observe that
$$
 1 = TU1 = F \cdot (U1\circ\vf)
$$
which shows that neither $F$ nor $U1\circ\vf$ can vanish in $\D$ and
\begin{equation}
 F=\frac{1}{U1\circ\vf}\,.
 \label{F}
\end{equation}
Now, let $g=Uz$ and $h=U(z^n)$. Then $Tg= z$ and
$$
 F (h\circ\vf) = Th = z^n = (Tg)^n = F^n (g^n\circ\vf) \,.
$$
It follows from \eqref{F} that
$$
 (h\circ \vf) ((U1)^{n-1}\circ\vf) = g^n\circ \vf\,.
$$
Since $\vf$ is not identically constant, we deduce from the uniqueness principle for analytic functions that $h (U1)^{n-1} = g^n$ holds throughout $\D$. Thus, in view of the removable singularities,
$$
  h = \frac{g^n}{(U1)^{n-1}} = \(\frac{g}{U1}\)^n \cdot U1
$$
also holds in all of $\D$. By the way $g$ and $h$ were defined, this means that \eqref{U-T} holds. Now Theorem~\ref{thm-wco-charact} implies that $U$ is a WCO.
\par\medskip
Next, knowing that $U=T_{G,\psi}$ for some $G\in\ch(\D)$ and some analytic self-map $\psi$ of $\D$, we find an explicit formula for $U$. Starting from
$$
 1 = U T 1 = U(F) = G \cdot (F\circ\psi)
$$
it is clear that $G$ does not vanish in $\D$. Using also the representation of $T=T_{F,\vf}$ and the fact that $f(z)=z$ is also a function in $X$, we obtain
$$
 z = U T z = U(F\cdot\vf) = G \cdot (F\circ\psi) (\vf\circ\psi) = UF \cdot (\vf\circ\psi) = \vf\circ\psi
$$
so $\vf\circ\psi=id$, the identity mapping of $\D$. By a similar reasoning, we also deduce that $\psi\circ \vf=id$. Therefore both $\vf$ and $\psi$ are bijective mappings of $\D$, hence they are disk automorphisms and mutually inverse.
\par
As for the symbols of $U=T^{-1}$, the above reasoning shows that
$\psi=\vf^{-1}$ and from \eqref{F} we get
$$
 G = U1 = (U1 \circ\vf) \circ\psi = \frac{1}{F\circ\psi} = \frac{1}{F\circ\vf^{-1}}\,.
$$
\par\medskip
(b) There are two implications to be proved.
\par\smallskip
\fbox{$\Rightarrow$} \ First suppose that $T=T_{F,\vf}$ is invertible. By what we have already proved in part (a) using only axioms~\textbf{[Ax1] - [Ax4]}, we know that $F$ does not vanish in $\D$, that $U=T^{-1}$ is also a WCO and we know that its symbols are given by \eqref{inv-symb}. So, it is only left to show that $1/F\in\cm (X)$, assuming also Axiom~\textbf{[Ax5]}.
\par
To this end, we first observe that $F\in\cm(X)$. By Axiom~\textbf{[Ax5]}, disk automorphisms induce bounded composition operators. Hence $C_\psi$ is bounded on $X$ and so is the operator $T_{F,\vf} C_\psi = M_F$. This shows that $F\in\cm(X)$.
\par
Similarly, considering $C_\vf$ instead of $C_\psi$, $U=T^{-1}$ instead of $T$ and $G=U1$ instead of $F=T1$, we see that $G\in\cm(X)$ as well.
\par
In order to show that $1/F\in\cm(X)$, let $f\in X$ be arbitrary. Since $C_\psi$ is bounded on $X$ it follows that $f\circ\psi\in X$, hence
$$
 G (f\circ\psi) = \frac{f\circ\psi}{F\circ\psi} \in X\,.
$$
Finally, since  $C_\vf$ is bounded on $X$, it follows that
$$
 \frac{f\circ\psi}{F\circ\psi}\circ\vf = \frac{f}{F} \in X\,.
$$
This shows that $1/F\in\cm(X)$.
\par\medskip
\fbox{$\Leftarrow$} \ Now assume that $T_{F,\vf}$ bounded on $X$, its composition symbol $\vf$ is an automorphism of $\D$, the multiplication symbol $\,F$ does not vanish in the disk, and $1/F\in\cm (X)$. We want to show that $T$ is invertible by checking that the operator $U=T_{G,\psi}$, with $G$ and $\psi$ given by \eqref{inv-symb}, acts boundedly on $X$ and $TU=UT=I$.
\par
In view of Axiom~\textbf{[Ax5]}, both operators $C_\vf$ and $C_\psi$ where $\psi=\vf^{-1}$, are bounded on $X$. Thus, if $f\in X$ then $f\circ\vf\in X$ and then also $(f\circ\vf)/F\in X$ since $1/F\in\cm (X)$. But $C_\psi$ maps $X$ to itself and therefore we also have that
$$
 \frac{f}{F\circ\psi} = \frac{f\circ\vf}{F}\circ\psi \in X\,.
$$
Since this holds for arbitrary $f\in X$, if follows that $G=1/(F\circ\psi)\in\cm (X)$. Now it is clear that the compositions
$$
 U T = M_G C_\psi M_F C_\vf\,, \qquad T U = M_F C_\vf M_G C_\psi
$$
make sense as operators mapping $X$ to itself, and a short and direct computation shows that both coincide with the identity operator $I$ on $X$.
\end{proof}
%%%%%%%%%%
\par
%%%%%%%%%%
\begin{ex} \label{ex-isom-unbded}
Consider the space $X=H^\infty_v$ with the non-radial weight $v(z)=|1-z|$. This space contains $H^\infty$ and also some unbounded functions such as $f(z)=\frac{1}{1-z}$. It is immediate that the point evaluations are bounded on $X$ (and uniformly bounded on compact sets), with
$$
 \frac12 \le \|\La_z\|\le \frac{1}{|1-z|}
$$
for all $z$ in $\D$ (for the first inequality, consider any non-zero constant function). By a standard normal family argument, it follows that $X$ is a Banach space. The shift operator is trivially bounded on $X$, so the space satisfies the axioms \textbf{[Ax1]}, \textbf{[Ax3]}, and \textbf{[Ax4]}. It can also be checked that it fails to satisfy \textbf{[Ax2]} and \textbf{[Ax5]} so we cannot expect Theorem~\ref{thm-invert-comp-disk} to hold automatically. In fact, some of the conclusions in part (b) do not follow, and even more can be said.
\par
It is easy to check that the rotations do not necessarily generate bounded operators on $X$. If $\vf(z)=-z$, then $C_\vf$ does not map $X$ into itself: the function given by $f(z)=\frac{1}{1-z}$ is in $X$ but $C_\vf f$ is not:
$$
 \|C_\vf f\|= \|f(-z)\| = \sup_{z\in\D} \left| \frac{1-z}{1+z} \right| = \infty\,.
$$
Also, the function $F$, being unbounded, cannot generate a bounded pointwise multiplier $M_F$ from $X$ into itself. Nonetheless, the weighted composition operator $T_{F,\vf}$ with
$$
 F(z)=\frac{1+z}{1-z}\,, \quad \vf(z)=-z
$$
is bounded on $X$. While such examples have appeared earlier in the literature, what seems remarkable about this situation is that our operator $T_{F,\vf}$ is even a surjective isometry and an involution:
$$
 \|T_{F,\vf}\| = \sup_{z\in\D} |1+z| |f(-z)| =  \sup_{w\in\D} |1-w| |f(w)| = \|f\|\,, \qquad T_{F,\vf}^{-1}=T_{F,\vf} \,.
$$
\end{ex}
%%%%%%%%%%
\par

%%%%%%%%%%%%%%%%%%%%%%%%%%%%%%%%%%%%%%%%%%%%%%%%%%%%%%%%%%%%%%%%%%%%%
\section{Invertibility in spaces on general domains that satisfy five  axioms}
 \label{sec-invert-5axioms}
%%%%%%%%%%%%%%%%%%%%%%%%%%%%%%%%%%%%%%%%%%%%%%%%%%%%%%%%%%%%%%%%%%%%%
\par
In this section we work in a more general context of Banach spaces of analytic functions on a general bounded planar domain (without any connectivity assumptions). It should be noted that while the axioms in \cite{Bo} are quite general they refer necessarily to the disk so they are not suited for this general context. Our axioms -when applied to the spaces on the disk- are still weaker than those assumed in \cite{HLNS} and hence our result is more general. The  main idea is to avoid the use of Carleson measures and the technicalities typical of any individual space while relying on the properties of functions which are near extremal for the point evaluations.
\par
It should be noted that in this section we only prove a theorem on invertibility, different from Theorem~\ref{thm-invert-comp-disk}. Since the monomials are very special and so is the disk as a domain, it is not at all clear how an analogue of Theorem~\ref{thm-wco-charact} would look like on a general bounded domain.
\par\smallskip
In what follows we consider Banach spaces $X\subset \ch(\Omega)$ that satisfy the following set of axioms (ordered following certain similarity with the previous section):
\par
\begin{description}
\item[A1]
 All point evaluation functionals $\La_z$ are bounded on $X$.
\item[A2]
  $f_0\in X$, where $f_0(z)\equiv 1$.
\item[A3]
 The shift operator is bounded on $X$.
\item[A4]
 For every function $f\in X$ we have
$\displaystyle{\frac{|f(z)|}{\|\La_z\|}\to 0}$ \ as \  $\mathrm{dist\,}(z,\pd\Omega)\to 0$.
\item[A5]
  Each automorphism of $\Omega$ induces a bounded composition operator	 in $X$.
\end{description}
\par
One easily notices the similarities between the odd-numbered axioms in the above list with the corresponding ones from the previous section. Note that Axiom~\textbf{(A2)} is much weaker than the second axiom required earlier while Axiom~\textbf{(A4)} is different. It generalizes the well-known fact that in Hilbert spaces the normalized reproducing kernels tend to zero weakly: if we denote by $K_z$ the reproducing kernel at $z$, then
$$
 \frac{|f(z)|}{\|\La_z\|}=\frac{|\langle f,K_z\rangle|}{\|\La_z\|}= |\langle f,\frac{K_z}{\|\La_z\|}\rangle\big|\to 0\,.
$$
\par
Here are some spaces that verify the above axioms.
\begin{itemize}
 \item
The analytic Besov spaces $B^p$, $1<p<\infty$, verify all of the above axioms. Most are checked readily; for Axiom~\textbf{(A4)}, see \cite{HW} or \cite{BV}, for example.
 \item
Another important family of spaces that satisfy all five axioms~\textbf{(A1) - (A5)} is that of the  mixed norm spaces $H(p,q,\alpha)$ with $q<\infty$. That Axiom~\textbf{(A4)} holds was proved in the first author's paper \cite{Ar} while the validity of Axiom~\textbf{(A5)} was shown in her joint recent preprint with Contreras and Rodr\'{\i}guez-Piazza \cite{ACR}.
 \item
Many $H^2 (\beta)$ spaces of the disk also satisfy the five axioms. We have already discussed the validity of the first three axioms in Subsection~\ref{subsec-wco-charact}. A sufficient (integral) condition for Axiom~\textbf{(A5)} to hold is given by \cite[Theorem~3.5]{CM} while \cite[Theorem~2.10]{CM} combined with \cite[Theorem~2.17]{CM} shows that Axiom~\textbf{(A4)} is satisfied whenever $\sum_{n=1}^\infty \beta (n)^{-2}=\infty$.
\end{itemize}
\par\smallskip
It is quite routine to check from our axioms that if a weighted composition operator acts on a space:
$$
 T_{F,\vf}:X\to X\,, \qquad T_{F,\vf}f=F\cdot (f\circ \vf),
$$
where $F$ is analytic in $\Omega$ and $\vf$ is an analytic self-map of $\Omega$, then the operator is actually bounded. This follows from the closed graph theorem since norm convergence implies pointwise convergence by Axiom~\textbf{(A1)}.
\par\medskip
Note that in this context it is not clear at all what result, if any, should constitute an analogue of Theorem~\ref{thm-wco-charact} as the geometry of the disk and the role played by the monomials $z^n$ are quite special while here we are working with general bounded domains.
Thus, we prove only one result in this section: a theorem on invertibility of a WCO based on the five axioms above.
\par
We first need a simple topological lemma which states essentially that when we delete part of a domain, we inevitably add some boundary.
\par
%%%%%%%%%%
\begin{lem} \label{lem-gen-top}
Let $\Omega$ be a bounded planar domain and $D$ a non-empty domain contained in $\Omega$, $D\neq \Omega$. Then $\pd D\cap \Omega\neq \emptyset$.
\end{lem}
%%%%%%%%%%
\begin{proof}
Since $\Omega$ is connected and $D$ is a proper non-empty subset of $\Omega$, it follows that $D$ is not closed in $\Omega$; that is, $D$ is a proper subset of $\overline{D}\cap \Omega$ (its closure in $\Omega$), where $\overline{D}$ denotes the closure of $D\,$ in $\C$. Hence,
$$
 \pd D\cap \Omega= (\overline{D}\setminus D)\cap \Omega= (\overline{D}\setminus \Omega)\cap D \neq \emptyset\,,
$$
and the proof is complete.
\end{proof}
%%%%%%%%%%
\par
The following purely complex analysis statement may be of interest by itself. It should be compared with \cite[Corollary~2.10, p.~25]{Po}, a known statement with much more restrictive hypotheses and a bit stronger conclusion. It shows that for nicely behaved self-maps of a domain the behavior inside is somehow controlled by the behavior near the boundary. As is usual, we will denote by $d(z,\pd\Omega)$ the distance of the point $z$ to the boundary of $\Omega$ and will often write $z\to\pd\Omega$ to denote the fact that $d(z,\pd\Omega)\to 0$.
%%%%%%%%%%
\begin{thm} \label{thm-self-map}
Let $\Omega$ be a bounded planar domain and $\vf$ a non-constant analytic self-map of $\/\Omega$ with the property that $d(\vf(z),\pd\Omega)\to 0$ whenever $d(z,\pd\Omega)\to 0$. Then $\vf(\Omega)=\Omega$.
\end{thm}
%%%%%%%%%%
\begin{proof}
Clearly, $\vf(\Omega)$ is a domain contained in $\Omega$. Suppose that $\vf(\Omega)\neq \Omega$. Then by Lemma~\ref{lem-gen-top}, $\Omega\cap \pd \vf(\Omega)\neq \emptyset$, and we can find a point $w_0\in \Omega\cap \pd \vf(\Omega)$. Note that $w_0 \not\in \vf(\Omega)$. Hence there exists a sequence of points $w_n\in \vf(\Omega)$ such that $w_n\to w_0$ and therefore also a sequence of points $z_n\in \Omega$ such that $w_n=\vf(z_n)\to w_0\in \Omega$, as $n\to\infty$. Since $\Omega$ is a bounded domain, some subsequence $z_{n_k}\to z_0\in\overline{\Omega}$. Of course, $w_{n_k}\to w_0\in\Omega$ and this will force a contradiction in both possible cases: $z_0\in\Omega$ and $z_0\in\pd\Omega$.
\par
Indeed, if $z_0\in\Omega$ then $w_0=\lim \vf(z_{n_k})=\vf(z_0)$, which is contradiction with the fact that $w_0\not\in \vf(\Omega)$. And if $z_0\in\pd\Omega$ then $z_{n_k}\to z_0$ means that $\lim_{k\to\infty} \mathrm{dist\,}(z_{n_k},\Omega)=0$ hence, by assumption, $\lim_{k\to\infty} \mathrm{dist\,}(\vf(z_{n_k}),\Omega)=0$. But since $\vf(z_{n_k})=w_0$, this means that $w_0\in\pd\Omega$, which is in contradiction with $w_0\in\Omega$.
\end{proof}
%%%%%%%%%%
\par\medskip
We are now ready to prove a theorem on invertibility for our set of five axioms. Since the axioms assumed here are much weaker than the ones in \cite{HLNS}, our result is more general even in the case when $\Omega=\D$. The result, of course, also generalizes a theorem for $H^2$ from \cite{G} whose proof has partially served as an inspiration although some entirely new techniques were required here.
\par
The reader will undoubtedly notice that the statement below resembles Theorem~\ref{thm-invert-comp-disk} to some extent. Just like in Theorem~\ref{thm-invert-comp-disk}, the condition $F\in\cm(X)$ is not needed in the proof of (b), and thus it is not listed among the assumptions (it will again follow from the remaining ones).
%%%%%%%%%%
\begin{thm} \label{thm-wco-invert}
Let $X \subset \ch(\Omega)$ be a Banach space which satisfies the set of axioms~\textbf{(A1) - (A4)} and suppose that the weighted composition operator $T_{F,\vf}$ is bounded in $X$.
\par
(a) If $T_{F,\vf}$ is invertible in $X$ then its composition symbol $\vf$ is an automorphism of \ $\Omega$ and the multiplication symbol $\,F$ does not vanish in $\Omega$.
\par
(b) If, in addition to the axioms listed, the space $X$ also satisfies Axiom~\textbf{(A5)} we have the following characterization.
\par
The WCO $T_{F,\vf}$ is invertible on $X$ if and only if its composition symbol $\vf$ is an automorphism of \ $\Omega$, the multiplication symbol $\,F$ does not vanish in $\Omega$, and $1/F\in\cm (X)$. If this is the case, then $F$ is also a self-multiplier of $X$ and the inverse operator is $T_{G,\psi}$, with the symbols given by the formula \eqref{inv-symb}, which formally reads as in the case of the disk.
\end{thm}
%%%%%%%%%%
\begin{proof}
(a) Suppose that $T_{F,\vf}$ is invertible. We divide the proof into a few steps.
\begin{itemize}
\item $\varphi$ is injective:
\par
First of all, the function $F$ is in $X$ since $f_0\in X$ with $f_0(z)\equiv 1$ by Axiom~\textbf{(A2)} and $F=T_{F,\varphi}f_0\in X$. Let $f_1(z)=z$. Axiom~\textbf{(A3)} implies that $f_1\cdot F\in X$. By assumption, $T_{F,\varphi}$ is invertible on $X$ and hence surjective. Therefore, there exists a function $g\in X$ such that
$$
 T_{F,\varphi}g=F\cdot(g\circ \varphi)=f_1\cdot F,
$$
that is, $g(\varphi(z))=z$ for $z\in\Omega.$ From here it follows that $\varphi$ is injective: if $\alpha$, $\beta\in\Omega$ are such that $\varphi(\alpha)=\varphi(\beta),$ then
$$
 \alpha=g(\varphi(\alpha))=g(\varphi(\beta))=\beta\,.
$$
\par
\item $\varphi$ is onto and, thus, an automorphism of $\Omega$, and also $F$ does not vanish in $\Omega$:
\par
Let $z\in\D$ be arbitrary. By Axiom~\textbf{(A1)}, both $\La_z$ and $\La_{\vf(z)}$ are bounded; moreover, $\|\La_z\|>0$ and $\|\La_{\vf(z)}\|>0$ in view of Axiom~\textbf{(A2)}. Observe that $\La_z T_{F,\vf} = F(z) \cdot \La_{\vf(z)}$. Together with the invertibility assumption on $T_{F,\vf}$, this yields
$$
 |F(z)|\cdot \|\La_{\varphi(z)}\| = \|\La_z T_{F,\vf}\| \ge \frac{\|\La_z\|}{\|T_{F,\vf}^{-1}\|}\,,
$$
whence for $m=\frac{1}{\|T_{F,\vf}^{-1}\|}$ we obtain
\begin{equation}
 \label{eq1}
 0<\frac{m}{\|\La_{\varphi(z)}\|}\le \frac{|F(z)|}{\|\La_z\|}\,.
\end{equation}
Since $z$ was arbitrary, we conclude that this holds for all $z\in \D$.
\par
By Axiom~\textbf{(A4)}, we have
$$
 \frac{|F(z)|}{\|\La_z\|}\to 0 \quad \mathrm{as } \quad z\to \partial\Omega\,,
$$
hence by \eqref{eq1}:
$$
 \|\La_{\varphi(z)}\|\to\infty \quad \mathrm{as } \quad z\to\partial\Omega\,.
$$
We will now see that this means that $\varphi(z)\to \partial\Omega$ as $z\to\partial\Omega.$ If this was not the case, we could find a sequence $\{z_n\}\subset\Omega$ such that $z_n\to\partial\Omega,$ $\|\La_{\varphi(z_n)}\|\to\infty$ and $\{\varphi(z_n)\}$ is contained in a compact subset of  $\Omega.$ For every function $f\in X$ and a compact subset $K$ of $\Omega$, the set of values $\{|\La_{\varphi(z)} f|\,\colon\,z\in K\}$ is bounded as a consequence of Axiom~\textbf{(A1)}. Hence by the uniform boundedness principle there exists a constant $C_K$ such that $\|\La_{\varphi(z_n)}\|\leq C_K$, which is absurd.
\par
Thus, $\varphi$ has the property claimed and this implies that it is a surjective self-map of $\Omega$ by Theorem~\ref{thm-self-map}.
\end{itemize}
\par\medskip
(b) As before, some of the conclusions follow from part (a) already proved. We prove the rest in two steps.
\begin{itemize}
\item $F$ is a self-multiplier of $X.$
\par
Since $\varphi$ is an automorphism of $\Omega$, by Axiom~\textbf{(A5)} the operator $C_{\varphi^{-1}}$ is bounded and
$$
 \|M_F f\|=\|T_{F,\varphi}C_{\varphi^{-1}}f\|\leq K \|f\|
$$
for some positive constant $K$ and all $f\in X$. Thus, $M_F$ is a bounded operator on $X$.
\par
\item $\frac{1}{F}$ is a self-multiplier of $X.$
\par
By \eqref{eq1}, for all $z\in\D$ we have
$$
 |F(z)|\geq m\frac{\|\La_z\|}{\|\La_{\varphi(z)}\|} \,.
$$
In order to bound $F$ from below, it suffices to see that
\begin{equation}
 \label{eq2}
 \|\La_{\varphi(z)}\|=\|\La_zC_\varphi\|\leq\|\La_z\|\|C_\varphi\|\,.
\end{equation}
It follows from here that
$$
 |F(z)|\ge m \frac{\|\La_z\|}{\|\La_{\varphi(z)}\|} \ge \frac{m}{\|C_\varphi\|}> 0\,.
$$
for every $z\in\Omega$ since the composition operator with symbol $\varphi$ is bounded by Axiom~\textbf{(A5)}. Note that $\|C_\vf\|>0$ by virtue of Axiom~\textbf{(A2)}. Hence $\frac{1}{F}$ is analytic in $\Omega.$ Since the operator
$T_{F,\varphi}$ is surjective, for each $f\in X$ there exists $g\in X$ such that
$$
 f=T_{F,\varphi}g=F\cdot(g\circ\varphi).
$$
Thus,
$$
 \frac{1}{F}f=g\circ\varphi\in X,
$$
since $C_\varphi$ is a bounded operator on $X.$ It follows that $M_{\frac{1}{F}}$ is also bounded on $X.$
\end{itemize}
\par\medskip
Conversely, if the following three assumptions are satisfied: the composition symbol $\vf$ is an automorphism of \ $\Omega$, the multiplication symbol $\,F$ does not vanish in $\Omega$, and $1/F\in\cm (X)$, we will show that $T_{F,\vf}$ is invertible and its inverse is given by the expected formula.
\par
If $\varphi$ is an automorphism of $\Omega$, so is its inverse $\varphi^{-1}$. By Axiom~\textbf{(A5)} the composition operator  $C_{\varphi^{-1}}$ is bounded on $X.$ The multiplication operator $M_{\frac{1}{F}}$ is bounded on $X$ by assumption, hence the operator $C_{\varphi^{-1}}M_{\frac{1}{F}}$ is bounded. Moreover, for each function $f$ in $X$ we have
$$
 T_{F,\varphi}C_{\varphi^{-1}}M_{\frac{1}{F}}f =T_{F,\varphi}\left(\frac{f}{F}\circ \varphi^{-1}\right) =F\cdot\left(\left(\frac{f}{F}\circ \varphi^{-1}\right)\circ \varphi\right) =f
$$
and also
$$
 C_{\varphi^{-1}}M_{\frac{1}{F}} T_{F,\varphi}f=C_{\varphi^{-1}} M_{\frac{1}{F}}\left(F\cdot(f\circ \varphi)\right) = (f\circ \varphi) \circ \varphi^{-1}=f\,.
$$
In other words, the operator $T_{F,\vf}$ is invertible and
$$
 T^{-1}_{F,\varphi} = C_{\varphi^{-1}}M_{\frac{1}{F}} = M_{\frac{1}{F\circ \varphi^{-1}}}C_{\varphi^{-1}} = T_{\frac{1}{F\circ \varphi^{-1}},\varphi^{-1}} \,.
$$
\end{proof}
%%%%%%%%%%
\par
The space from Example~\ref{ex-isom-unbded} is again relevant here. This time it satisfies the axioms \textbf{(A1)}, \textbf{(A2)}, and \textbf{(A3)}. It can also be checked that it fails to satisfy \textbf{(A4)} and \textbf{(A5)}; for example, in view of our observations in Example~\ref{ex-isom-unbded}, we see that in general
$$
 \frac{|f(z)|}{\|\Lambda_z\|}\ge 2 |f(z)|\not\to 0\,, \qquad |z|\to 1^-\,.
$$
Thus, we cannot expect Theorem~\ref{thm-wco-invert} to hold automatically. In fact, we already know from Example~\ref{ex-isom-unbded} that not all conclusions in part (b) can hold since $1/F$ is an unbounded function and therefore cannot multiply $X$ into itself.
\par\smallskip
\textsc{Acknowledgments}. {\small The authors would like to thank Manuel Contreras and Ignacio Uriarte-Tuero for some useful comments and especially to the referee for various relevant suggestions for simplifying several arguments.}

%%%%%%%%%%%%%%%%%%%%%%%%%%%%%%%%%%%%%%%%%%%%%%%%%%%%%%%%%%%%%%%%

\end{document}